\documentclass[12pt]{amsart}

\usepackage{amsmath, amssymb, bbm}
\DeclareMathOperator{\trace}{trace}
\newtheorem{theorem}{Theorem}[section]
\newtheorem{lemma}[theorem]{Lemma}
\newtheorem{prop}[theorem]{Proposition}
\newtheorem{cor}[theorem]{Corollary}

\theoremstyle{definition}
\newtheorem{defi}[theorem]{Definition}

\theoremstyle{remark}
\newtheorem{remark}[theorem]{Remark}

\numberwithin{equation}{section}

\newcommand{\rd}{{\mathbb R^d}}

\newcommand{\Ker}{\operatorname{Ker}}

\allowdisplaybreaks

\begin{document}
\sloppy
\title[Hausdorff dimension of operator scaling stable random sheets]{The Hausdorff dimension of\\ the graph of operator scaling stable random sheets} 

\author{Ercan S\"onmez}
\address{Ercan S\"onmez, Mathematisches Institut, Heinrich-Heine-Universit\"at D\"usseldorf, Universit\"atsstr. 1, D-40225 D\"usseldorf, Germany}
\email{ercan.soenmez\@@{}hhu.de} 

\date{\today}

\begin{abstract}
We consider operator scaling $\alpha$-stable random sheets, which were introduced in \cite{Hoffm}. The idea behind such fields is to combine the properties of operator scaling $\alpha$-stable random fields introduced in \cite{BMS} and fractional Brownian sheets introduced in \cite{Kamont}. Based on the results derived in \cite{Hoffm}, we determine the box-counting dimension and the Hausdorff dimension of the graph of a trajectory of such fields over a non-degenerate cube $I \subset \mathbb{R}^d$.
\end{abstract}

\thanks{This work has been supported by Deutsche Forschungsgemeinschaft (DFG) under grant KE1741/ 6-1}
\maketitle

\baselineskip=18pt

\section{Introduction}

In this paper, we consider a harmonizable operator scaling $\alpha$-stable random sheet as introduced in \cite{Hoffm}. The main idea is to combine the properties of operator scaling $\alpha$-stable random fields and fractional Brownian sheets in order to obtain a more general class of random fields. Let us recall that a scalar valued random field $\{ X(x) : x \in \mathbb{R}^d \}$ is said to be \textit{operator scaling} for some matrix $E \in \mathbb{R}^{d \times d}$ and some $H>0$ if

\begin{equation}\label{OSSRF}
	\{ X( c^Ex) : x \in \mathbb{R}^d \} \stackrel{\rm f.d.}{=} \{ c^H X(x) : x \in \mathbb{R}^d \}  \quad \text{for all } c>0,
\end{equation}
where $\stackrel{\rm f.d.}{=}$ means equality of all finite-dimensional  marginal distributions and, as usual, $c^E = \sum_{k=0}^{\infty} \frac{(\log c)^k}{k!} E^k$ is the matrix exponential. These fields can be regarded as an anisotropic generalization of self-similar random fields (see, e.g., \cite{EmbrMaeh}), whereas the fractional Brownian sheet $\{ B_{H_1, \ldots, H_d} (x) : x \in \mathbb{R}^d \}$ with Hurst indices $H_1, \ldots, H_d >0$ can be seen as an anisotropic generalization of the well-known fractional Brownian field (see, e.g., \cite{Kolm}) and satisfies the scaling property \newpage
\begin{align*}
\{ B_{H_1, \ldots, H_d} (c_1x_1, \ldots, c_dx_d) : x =(x_1, \ldots, x_d) \in \mathbb{R}^d \} \\
\stackrel{\rm f.d.}{=} \{ c_1^{H_1} \ldots c_d^{H_d} B_{H_1, \ldots, H_d} (x) : x  \in \mathbb{R}^d \}
\end{align*}
for all constants $c_1, \ldots, c_d >0$. See \cite{AychXiao, Herbin, XiaoZhang} and the references therein for more information on the fractional Brownian sheet.

Throughout this paper, let $d= \sum_{j=1}^m d_j$ for some $m \in \mathbb{N}$ and $\tilde{E}_j \in \mathbb{R}^{d_j \times d_j}$, $j=1, \ldots, m$ be matrices with positive real parts of its eigenvalues. We define matrices $E_1, \ldots, E_m \in \mathbb{R}^{d \times d}$ as
\begin{equation*}
E_j = \left(
\begin{array}{ccccccccccccccccccccccc}
0 & & & & & & 0  \\
 &  \ddots & & & &  \\
 &  & 0 & & &  \\
 & & & \tilde{E}_j & & & \\
& & & &  0 &  &   \\
 & & & & & \ddots &  \\
0 & & & & & & 0
\end{array} \right) .
\end{equation*}
Further, we define the block diagonal matrix $E \in \mathbb{R}^{d \times d}$ as
\begin{equation*}
E = \sum_{j=1}^m E_j = \left(
\begin{array}{ccccccccccccccccccccccc}
E_1 & & 0  \\
 &  \ddots &   \\
0 & & E_m
\end{array} \right) .
\end{equation*}
In analogy to \cite[Definition 1.1.1]{Hoffm}, a random field $\{ X(x) : x \in \mathbb{R}^d \}$ is called \textit{operator scaling stable random sheet} if for some $H_1, \ldots, H_m >0$ we have
\begin{equation}\label{OSSRS}
	\{ X( c^{E_j}x) : x \in \mathbb{R}^d \} \stackrel{\rm f.d.}{=} \{ c^{H_j} X(x) : x \in \mathbb{R}^d \}
\end{equation}
for all $c>0$ and $j =1, \ldots, m$. Note that, by applying \eqref{OSSRS} iteratively, any operator scaling stable random sheet is also operator scaling for the matrix $E$ and the exponent $H= \sum_{j=1}^m H_j$ in the sense of \eqref{OSSRF}. Further, note that this definition is indeed a generalization of operator scaling random fields, since for $m=1, d =d_1$ and $E=E_1 = \tilde{E}_1$ \eqref{OSSRS} coincides with the definition introduced in \cite{BMS}. The existence of a random field satisfying \eqref{OSSRS} is guaranteed, since, taking $E_j=d_j=1$ for $j =1, \ldots, m$, an example is given by the fractional Brownian sheet. Operator scaling stable random sheets have been proven to be quite flexible in modeling physical phenomena and can be applied in order to extend the well-known Cahn-Hilliard phase-field model. See \cite{AHSW} and the references therein for more information.

Random fields satisfying a scaling property such as \eqref{OSSRF} or \eqref{OSSRS} are very popular in modeling, see \cite{Levy, WillPaxTaqq} and the references in \cite{BL} for some applications. Most of these fields are Gaussian. However, Gaussian fields are not always flexible for example in modeling heavy tail phenomena. For this purpose, $\alpha$-stable random fields have been introduced. See \cite{SamorodTaqq} for a good introduction to $\alpha$-stable random fields.

Using a moving average and a harmonizable representation, the authors in \cite{BMS} defined and analyzed two different classes of symmetric $\alpha$-stable random fields satisfying \eqref{OSSRF}. In the Gaussian case $\alpha = 2$ results about certain H\"older conditions and the Hausdorff dimension have been obtained. Similar results in the stable case $\alpha \in (0,2)$ have been derived in \cite{BL} for the harmonizable representation. Following the outline in \cite{BMS, BL}, this two classes were generalized to random fields satisfying \eqref{OSSRS} in \cite{Hoffm}. Similar results about H\"older conditions have been obtained. The fields constructed in \cite{BMS} have stationary increments, i.e. they satisfy
$$ \{X( x+h) - X(h) : x \in \mathbb{R}^d \} \stackrel{\rm f.d.}{=} \{X( x) : x \in \mathbb{R}^d \} \quad \text{for all } h \in \mathbb{R}^d .$$
This property has been proven to be quite useful in studying the sample path properties. However, the property of stationary increments is no more true for the fields constructed in \cite{Hoffm}. The absence of this property seems to be one of the main difficulties in determining results about the Hausdorff dimension.

Another main tool in studying sample paths of operator scaling stable random sheets are polar coordinates with respect to the matrices $E_j, j=1, \ldots, m$ introduced in \cite{MeersScheff} and used in \cite{BMS, BL}. If $\{X( x) : x \in \mathbb{R}^d \}$ is an operator scaling symmetric $\alpha$-stable random sheet with $\alpha = 2$, using \eqref{OSSRS} one can write the variance of $X( x), x \in \mathbb{R}^d$, as
$$ \mathbb{E} [X^2 (x) ] = \tau_{E_j} (x) ^{2H_j} \mathbb{E} [X^2 \big( l_{E_j} (x) \big) ],$$
where $\tau_{E_j} (x)$ is the radial part of $x$ with respect to $E_j$ and $l_{E_j} (x)$ is its polar part. Therefore, in the Gaussian case information about the behavior of the polar coordinates $\big( \tau_{E_j} (x), l_{E_j} (x) \big)$ contains information about the sample path regularity. This property also holds in the stable case $\alpha \in (0,2)$.

This paper is organized as follows. In Section 2, we introduce the main tools we need for the study in this paper. Since our main focus will be on Hausdorff dimension and box-counting dimension, we recall their definition and some related results in Section 2.1. Section 2.2 is devoted to a spectral decomposition result from \cite{MeersScheff} and Section 2.3 is about the change to polar coordinates with respect to scaling matrices. In Section 3, we present the results derived in \cite{Hoffm} on harmonizable and moving average representations of operator scaling $\alpha$-stable random sheets. Here, we will only focus on a harmonizable representation. Based on this results and taking into account methods used in \cite{Aych, BMS, BL, Xiao1}, in Section 4 we present our main results on the Hausdorff dimension and box-counting dimension of the graph of harmonizable operator scaling stable random sheets.

\section{Preliminaries}

\subsection{Hausdorff dimension and box-counting dimension} 

Let $B \subset \mathbb{R}^d$ and $s \geq 0$. The $s$-dimensional Hausdorff measure of $B$ is defined by
\begin{equation*}
	\mathcal{H}^{s}(B)=\lim_{\varepsilon\downarrow0}\inf\left\{\sum_{k=1}^{\infty}|B_{k}|^{s}:\,B\subset\bigcup_{k=1}^{\infty}B_{k},\; |B_{k}|\leq \varepsilon\right\}.
\end{equation*}
One can easily show that $\mathcal{H}^{s}(B) < \infty$ implies $\mathcal{H}^{t}(B) = 0$ for all $t>s$ (see \cite[Chapter 2.2]{Fal}). Thus, there exists a critical value, denoted by $\dim_{\mathcal{H}} B$, such that
\begin{equation*}
\dim_{\mathcal{H}} B=\inf\{s \geq 0:\,\mathcal{H}^{s}(B)=0\}=\sup\{s \geq 0:\,\mathcal{H}^{s}(B)=\infty\}.
\end{equation*}
$\dim_{\mathcal{H}} B$ is called the Hausdorff dimension of $B$. Now, in addition assume that $B$ is non-empty and bounded. For $\varepsilon > 0$, let
$$ M (B, \varepsilon ) = \min \{ k \in \mathbb{N} : B \subset \bigcup_{i=1}^k B_\varepsilon (x_i) \} $$
be the smallest number of balls of diameter at most $\varepsilon$ which can cover $B$. The lower and upper box-counting dimensions of $B$ respectively are defined as
\begin{align*}
 \underline{\dim_{\mathcal{B}}} B & = \liminf_{ \varepsilon \downarrow 0} \frac{\log M(B, \varepsilon)}{- \log \varepsilon}  \\
 \overline{\dim_{\mathcal{B}}} B & = \limsup_{ \varepsilon \downarrow 0} \frac{\log M(B, \varepsilon)}{- \log \varepsilon}.
\end{align*}
If these are equal we refer to the common value as the box-counting dimension of $B$, denoted by $\dim_{\mathcal{B}} B$, i.e.
$$ \dim_{\mathcal{B}} B = \lim_{ \varepsilon \downarrow 0} \frac{\log M(B, \varepsilon)}{- \log \varepsilon} .$$
See \cite[Chapter 3.1]{Fal} for equivalent definitions of $\dim_{\mathcal{B}} B$.

In order to determine the Hausdorff dimension of $B$, one usually gives an upper bound and a lower bound for $\dim_{\mathcal{H}} B$. The following result will be useful for us in order to find an upper bound for $\dim_{\mathcal{H}} B$. See \cite[Corollary 11.2]{Fal}.
\begin{lemma}\label{upperbound}
Let $f : \mathbb{R}^d \mapsto \mathbb{R}$ and denote by $G_f \big( [0,1]^d \big) = \{ \big( x, f(x) \big) : x \in [0,1]^d \} \subset \mathbb{R}^{d+1}$ the graph of $f$ over the unit cube. Suppose that $f$ satisfies a uniform H\"older condition of order $s$ for some $s \in (0,1]$, i.e. there exists a constant $C>0$ such that
$$ | f(t) - f(u) | \leq C \| t-u \|^s $$
for all $t,u \in [0,1]^d$. Then we have
$$ \dim_{\mathcal{H}} G_f \big( [0,1]^d \big) \leq \overline{\dim_{\mathcal{B}}} G_f \big( [0,1]^d \big) \leq d+1-s.$$
\end{lemma}

Now let $B\subset\rd$ be a Borel set and let $\mathcal M^1 \big( B, \mathcal{B} (B) \big)$ be the set of probability measures on $B$. For $s>0$ the $s$-energy of $\mu \in \mathcal M^1 \big( B, \mathcal{B} (B) \big)$ is defined by
$$I_s(\mu)=\int_B\int_B\frac{\mu(dx)\mu(dy)}{\|x-y\|^s}.$$
In order to find a lower bound for $\dim_{\mathcal{H}} B$, by Frostman's theorem (see, e.g., \cite{Adler, Fal, Kahane}), it suffices to show that there exists a probabilty measure $\mu \in \mathcal M^1 \big( B, \mathcal{B} (B) \big)$ with finite $s$-energy. To be more precise, if there exists $\mu \in \mathcal M^1 \big( B, \mathcal{B} (B) \big)$ with
\begin{equation}\label{fse}
I_s (\mu) < \infty
\end{equation}
for some $s>0$, then we have $\dim_{\mathcal{H}} B \geq s$ (see also \cite[Theorem 4.27]{MoertPeres}). If we consider the random graph $G_X \big( [0,1]^d \big) = \{ \big( x, X(x) \big) : x \in [0,1]^d \}$ of a random field $\{ X(x) : x \in \mathbb{R}^d \}$ over the unit cube, a typical choice of a (random) probabilty measure 
$$\mu \in \mathcal M^1 \Bigg( X \big( [0,1]^d \big) , \mathcal{B} \Big( X \big( [0,1]^d \big) \Big) \Bigg) $$
is the occupation measure given by
$$ \mu (F) = \int_{[0,1]^d} \mathbbm{1}_{ \big\{ ( x, X(x) ) \in F \big\} } dx, \quad \text{for }F \in \mathcal{B} \big( [0,1]^{d+1} \big) . $$
Then, by a monotone class argument, it is easy to see that \eqref{fse} is equivalent to
$$ \int_{ [0,1]^d \times [0,1]^d } \big( \| x-y \|^2 + |X(x) - X(y)|^2 \big) ^{-\frac{s}{2}} dx dy < \infty, $$
which almost surely follows from
\begin{equation}\label{frostman}
\int_{ [0,1]^d \times [0,1]^d } \mathbb{E} \Big[ \big( \| x-y \|^2 + |X(x) - X(y)|^2 \big) ^{-\frac{s}{2}} \Big] dx dy < \infty.
\end{equation}

\subsection{Spectral decomposition}

Let $A \in \mathbb{R}^{d \times d}$ be a matrix with $p$ distinct positive real parts of its eigenvalues $0 < a_1 < \ldots < a_p$ for some $p \leq d$. Factor the minimal polynomial of $A$ into $f_1, \ldots, f_p$, where all roots of $f_i$ have real part equal to $a_i$, and define $V_{i}=\Ker\big( f_{i}(A) \big)$. Then, by \cite[Theorem 2.1.14]{MeersScheff},
$$\rd=V_{1}\oplus\ldots\oplus V_{p}$$
is a direct sum decomposition, i.e. we can write any $x \in \mathbb{R}^d$ uniquely as
$$ x = x_1 + \ldots + x_p$$
for $x_i \in V_i$, $1 \leq i \leq p$. Further, we can choose an inner product on $\rd$ such that the subspaces $V_1, \ldots, V_p$ are mutually orthogonal. Throughout this paper, for any $x \in \rd$ we will choose $ \| x \| = \langle x,x\rangle^{1/2}$ as the corresponding Euclidean norm.

\subsection{Polar coordinates}

We now recall the results about the change to polar coordinates used in \cite{BMS, BL}. As before, let $A \in \mathbb{R}^{d \times d}$ be a matrix with positive real parts of its eigenvalues $0 < a_1 < \ldots < a_p$ for some $p \leq d$. According to \cite[Section 3]{BL}, one can write any $x \in \rd \setminus \{0 \}$ uniquely as
\begin{equation}\label{pc}
x = \tau_A (x)^A l_A (x) ,
\end{equation}
where $\tau_A (x) >0$ is called the radial part of $x$ with respect to $A$ and $l_A(x) \in \{ x \in \rd : \tau_A (x) = 1 \}$ is called the direction. It is clear that $\tau_A (x) \rightarrow \infty$ as $\| x\| \rightarrow \infty$ and $\tau_A (x) \rightarrow 0$ as $\| x\| \rightarrow 0$. Further, one can extend $\tau_A( \cdot )$ continuously to $\rd$ by setting $\tau_A(0) = 0$. Note that, by \eqref{pc}, it is straightforward to see that $\tau_A( \cdot )$ satisfies
$$\tau_A( c^A x) = c \cdot \tau_A(x) \quad \text{for all } c>0.$$
Such functions are called $A$-homogeneous.

Let us recall a result about bounds on the growth rate of $\tau_A( \cdot )$ in terms of $a_1, \ldots, a_p$ established in \cite{BMS}.

\begin{lemma}\label{pcbounds}
Let $\varepsilon >0$ be small enough. Then there exist constants $K_1, \ldots, K_4 >0$ such that
$$ K_1 \| x \| ^{\frac{1}{a_1}+\varepsilon} \leq \tau_A(x) \leq   K_2 \| x \| ^{\frac{1}{a_p}-\varepsilon}$$
for all $x$ with $\tau_A(x) \leq 1$, and
$$ K_3 \| x \| ^{\frac{1}{a_p}-\varepsilon} \leq \tau_A(x) \leq   K_4 \| x \| ^{\frac{1}{a_1}+\varepsilon}$$
for all $x$ with $\tau_A(x) \geq 1.$
\end{lemma}
We remark that the bounds on the growth rate of $\tau_A (\cdot )$ have been improved in \cite[Proposition 3.3]{BL}, but the bounds given in Lemma \ref{pcbounds} suffice for our purposes.

\section{Harmonizable operator scaling random sheets}

We consider harmonizable operator scaling stable random sheets defined in \cite{Hoffm} and present some related results established in \cite{Hoffm}. Most of these will also follow from the results derived in \cite{BMS, BL}. Throughout this paper, for $j =1, \ldots, m$ assume that the real parts of the eigenvalues of $\tilde{E}_j$ are given by $0< a_1^j< \ldots < a_{p_j}^j$ for some $p_j \leq d_j$. Let $q_j =  \trace (\tilde{E}_j ) $. Suppose that $\psi _j : \mathbb{R}^{d_j} \mapsto [0, \infty ) $ are continuous $\tilde{E}_j^T$-homogeneous functions, which means according to \cite[Definition 2.6]{BMS} that
$$\psi _j ( c^{\tilde{E}_j^T} x) = c \psi (x) \quad \text{for all } c>0.$$
Moreover, we assume that $\psi_j(x) \neq 0$ for $x \neq 0$. See \cite{BMS, BL} for various examples of such functions.

Let $0 \leq \alpha \leq 2$ and $W_{\alpha} (d \xi)$ be a complex isotropic symmetric $\alpha$-stable random measure on $\rd$ with Lebesgue control measure (see \cite[Chaper 6.3]{SamorodTaqq}).

\begin{theorem}\label{existence}
For any vector $x \in \rd$ let $x= (x_1, \ldots, x_m) \in \mathbb{R}^{d_1} \times \ldots \times \mathbb{R}^{d_m} = \rd$.  The random field
\begin{equation}\label{harmonizable}
X_{\alpha} (x) = \operatorname{Re} \int_{\rd} \prod_{j=1}^m ( e^{i \langle x_j, \xi_j \rangle} -1 ) \psi_j (\xi_j) ^{-H_j - \frac{q_j}{\alpha}} W_{\alpha} (d \xi) , \quad x \in \rd
\end{equation}
exists and is stochastically continuous if and only if $H_j \in (0, a_1^j)$ for all $j=1, \ldots, m$.
\end{theorem}

\begin{proof}
This result has been proven in detail in \cite{Hoffm}, but it also follows as an easy consequence of \cite[Theorem 4.1]{BMS}. By the definition of stable integrals (see \cite{SamorodTaqq}), $X_\alpha (x)$ exists if and only if
$$ \Gamma_\alpha (x) = \int _{\rd} \prod_{j=1}^m | e^{i \langle x_j, \xi_j \rangle} -1 | ^\alpha \psi_j (\xi_j) ^{-\alpha H_j - q_j} d \xi < \infty ,$$
but this is equivalent to
$$ \Gamma^j_\alpha (x) = \int _{\mathbb{R}^{d_j}} | e^{i \langle x_j, \xi_j \rangle} -1 | ^\alpha \psi_j (\xi_j) ^{-\alpha H_j - q_j} d \xi_j < \infty ,$$
for all $j=1, \ldots, m$. Since, in \cite[Theorem 4.1]{BMS}, it is shown that $\Gamma_\alpha^j (x)$ is finite if and only if $H_j \in (0, a_1^j)$ the statement follows. The stochastic continuity can be deduced similarly as a consequence of \cite[Theorem 4.1]{BMS}.
\end{proof}

Note that from \eqref{harmonizable} it follows that $X_\alpha (x) = 0$ for all $x= (x_1, \ldots, x_m) \in \mathbb{R}^{d_1} \times \ldots \times \mathbb{R}^{d_m} = \rd$ such that $x_j = 0$ for at least one $j \in \{ 1, \ldots, m \}$. 

The following result has been established in \cite[Corollary 4.2.1]{Hoffm}. The proof is carried out as the proof of \cite[Corollary 4.2 (a)]{BMS} via characteristic functions of stable integrals and by noting that $c^{E_j} x = (x_1, \ldots, x_{j-1}, c^{\tilde{E}_j} x_j, x_{j+1}, \ldots, x_m )$ for all $c>0$ and $x = (x_1, \ldots, x_m ) \in \mathbb{R}^{d_1} \times \ldots \times \mathbb{R}^{d_m} = \rd$.

\begin{cor}\label{operatorscaling}
Under the conditions of Theorem \ref{existence}, the random field $\{ X_\alpha (x) : x \in \rd \}$ is operator scaling in the sense of \eqref{OSSRS}, that is, for any $c>0$
\begin{equation}\label{operatorscalingharm}
\{ X (c^{E_j} x) : x \in \rd \} \stackrel{\rm f.d.}{=} \{ c^{H_j} X(x) : x \in \rd \}.
\end{equation}
\end{cor}

As we shall see below, fractional Brownian sheets fall into the class of random fields given by \eqref{harmonizable}. It is known that a fractional Brownian sheet does not have stationary increments. Thus, in general, a random field given by \eqref{harmonizable} does not possess stationary increments. But it satisfies a slightly weaker property, as the following statement shows.

\begin{cor}\label{stationary}
Let $x = (x_1, \ldots, x_m ) \in \mathbb{R}^{d_1} \times \ldots \times \mathbb{R}^{d_m}$. Under the conditions of Theorem \ref{existence}, for any $h \in \mathbb{R}^{d_j}$, $j =1, \ldots, m$ the random field $\{ X_\alpha (x) : x \in \rd \}$ satisfies
\begin{align*}
X_\alpha (x_1, \ldots, x_{j-1}, x_j & +h, x_{j+1}, \ldots, x_m ) - X_\alpha (x_1, \ldots, x_{j-1}, h, x_{j+1}, \ldots, x_m ) \\
& \stackrel{\rm d}{=}  X_\alpha (x_1, \ldots, x_{j-1}, h, x_{j+1}, \ldots, x_m ),
\end{align*}
where $\stackrel{\rm d}{=}$ means equality in distribution.
\end{cor}

\begin{proof}
This result has been established in \cite[Corollary 4.2.2]{Hoffm} and is proven similarly to \cite[Corollary 4.2 (b)]{BMS}.
\end{proof}

As in \cite{BMS, BL}, Hoffmann \cite{Hoffm} is looking for global and directional H\"older exponents of the random fields defined in \eqref{harmonizable}. In this paper, we will derive global H\"older critical exponents. Following \cite[Definition 5]{BonEstr}, $\beta \in (0,1)$ is said to be the H\"older critical exponent of the random field $\{ X(x) : x \in \rd \}$, if there exists a modification $X^*$ of $X$ such that for any $s \in (0, \beta)$ the sample paths of $X^*$ satisfy almost surely a uniform H\"older condition of order $s$ on any compact set $I \subset \rd$, i.e. there exists a positive and finite random variable $Z$ such that almost surely
\begin{equation}\label{hoeldercondition}
| X^*(x) - X^*(y) | \leq Z \| x - y \| ^s \quad \text{for all } x,y \in I,
\end{equation}
whereas, for any $s \in ( \beta ,1)$, \eqref{hoeldercondition} almost surely fails.

In order to derive the condition \eqref{hoeldercondition}, Hoffmann \cite{Hoffm} proposes the following definition.

\begin{defi}\label{subdef}
Let the conditions of Theorem \ref{existence} hold and fix $u = (u_1, \ldots, u_m ) \in \mathbb{R}^{d_1} \times \ldots \times \mathbb{R}^{d_m} = \rd$. For all $j =1, \ldots, m$ define the random field $\{ \tilde{X}_\alpha^{j,u} (x) : x \in \mathbb{R}^{d_j} \}$ as
\begin{equation}\label{sub}
\tilde{X}_\alpha^{j,u} (x) = X_\alpha (u_1, \ldots, u_{j-1}, x, u_{j+1}, \ldots, u_m ) ,  x \in \mathbb{R}^{d_j}.
\end{equation}
\end{defi}

Note that the random fields $\tilde{X}_\alpha^{j,u}$ defined in \eqref{sub} depend on the vector $u = (u_1, \ldots, u_m ) \in \mathbb{R}^{d_1} \times \ldots \times \mathbb{R}^{d_m}$, and if $u_i = 0$ for at least one $i \in \{1, \ldots, m\}, i \neq j$, by \eqref{harmonizable}, one ends up with the trivial random field given by $\tilde{X}_\alpha^{j,u} (x)= 0$ for all $x \in \mathbb{R}^{d_j}$. Moreover, by Corollary \ref{operatorscaling}, $\tilde{X}_\alpha^{j,u}$ is operator scaling in the sense of \eqref{OSSRF} for the matrix $\tilde{E}_j$ and the exponent $H_j$, since for all $c>0$
\begin{align*}
& \{ \tilde{X}_\alpha^{j,u} ( c^{\tilde{E}_j} x ) : x \in \mathbb{R}^{d_j} \} \\
& \stackrel{\rm def.}{=} \{ X_\alpha \big( c^{E_j} (u_1, \ldots, u_{j-1}, x, u_{j+1}, \ldots, u_m) \big) : x \in \mathbb{R}^{d_j} \} \\
& \stackrel{\rm f.d.}{=} \{ c^{H_j} X_\alpha (u_1, \ldots, u_{j-1}, x, u_{j+1}, \ldots, u_m) : x \in \mathbb{R}^{d_j} \} \\
& \stackrel{\rm def.}{=} \{ c^{H_j} \tilde{X}_\alpha^{j,u} ( x ) : x \in \mathbb{R}^{d_j} \} .
\end{align*}
Further, by Corollary \ref{stationary}, one can show that $\tilde{X}_\alpha^{j,u}$ has stationary increments, i.e. for any $h \in \mathbb{R}^{d_j}$
\begin{align*}
\{ \tilde{X}_\alpha^{j,u} (x+h) - \tilde{X}_\alpha^{j,u} (h) : x \in \mathbb{R}^{d_j} \} \stackrel{\rm f.d.}{=}  \{ \tilde{X}_\alpha^{j,u} (x) : x \in \mathbb{R}^{d_j} \} .
\end{align*}
In addition, as a consequence of Theorem \ref{existence}, one also obtains that
$$ \tilde{\Gamma}_\alpha ^{j,u} (x) := \mathbb{E} \big[ \tilde{X}_\alpha^{j,u} (x) \big] < \infty \quad \text{for all } x \in \mathbb{R}^{d_j},$$
and, in particular, there exists $0<M< \infty$ such that $\tilde{\Gamma}_\alpha ^{j,u} (x) \leq M$ for all $x \in \mathbb{R}^{d_j}$ with $\tau_{\tilde{E}_j} (x) = 1$. Thus, the random fields defined in \eqref{sub} satisfy all the properties that are used in \cite{BMS, BL} in order to determine conditions of the form \eqref{hoeldercondition}. Using all this properties Hoffmann \cite{Hoffm} carried this out in detail and derived H\"older exponents for the random fields in \eqref{sub}. More precisely, he showed the following.

\begin{lemma}\label{exponentsub}
Assume that the conditions of Theorem \ref{existence} hold and let $\tilde{X}_\alpha ^{j,u} (x)$ be defined as in Definition \ref{subdef}. Then there exist a positive and finite random variable $Z_j$ and a continuous modification of $\tilde{X}_\alpha ^{j,u} (x)$ such that for any $s \in (0, \frac{H_j}{a_{p_j}^j} )$ the uniform H\"older condition \eqref{hoeldercondition} holds almost surely.
\end{lemma}

As an easy consequence of Lemma \ref{exponentsub}, Hoffmann \cite{Hoffm} obtained the following result.

\begin{cor}\label{exponent}
Under the assumptions of Theorem \ref{existence}, there exist a positive and finite random variable $Z$ and a continuous modification of $X_\alpha$ such that for any $s \in (0, \min_{1 \leq j \leq m} \frac{H_j}{a_{p_j}^j} )$ the uniform H\"older condition \eqref{hoeldercondition} holds almost surely.
\end{cor}

\begin{proof}
For $j=1, \ldots, m$, let $\tilde{X}_\alpha ^{j,u}$ be as in \eqref{sub} and $0<s< \min_{1 \leq j \leq m} \frac{H_j}{a_{p_j}^j}$. Then, by Lemma \ref{exponentsub}, there exist positive and finite random variables $Z_1, \ldots, Z_m$ such that \eqref{hoeldercondition} holds almost surely for any $j=1, \ldots, m$ and any continuous modification of $\tilde{X}_\alpha ^{j,u}$. Then the statement easily follows, by noting that the inequality
\begin{align*}
& | X_\alpha (x) - X_\alpha (y) | \\
& \leq \sum_{i=1}^m | X_\alpha (x_1, \ldots, x_{i-1}, x_i, y_{i+1}, \ldots, y_m) - X_\alpha (x_1, \ldots, x_{i-1}, y_i, y_{i+1}, \ldots, y_m) |
\end{align*}
holds for all $x= (x_1, \ldots, x_m)$ and $y= (y_1, \ldots, y_m) \in \mathbb{R}^{d_1} \times \ldots \times \mathbb{R}^{d_m}$ with the convention that
$$X_\alpha (x_1, \ldots, x_{i-1}, y_i, y_{i+1}, \ldots, y_m) = X_\alpha(y) $$
for $i=1$ and
$$X_\alpha (x_1, \ldots, x_{i-1}, x_i, y_{i+1}, \ldots, y_m) = X_\alpha(x) $$
for $i=m$.
\end{proof}

We remark that Corollary \ref{exponent} is not a statement about critical H\"older exponents. However, as a consequence of Theorem \ref{hausdorff} below, we will see that any continuous version of $X_\alpha$ admits $\min_{1 \leq j \leq m} \frac{H_j}{a_{p_j}^j}$ as the critical exponent.

\section{Hausdorff dimension}

We now state our main result on the Hausdorff and box-counting dimension of the graph of $X_\alpha$ defined in \eqref{harmonizable}.

\begin{theorem}\label{hausdorff}
Suppose that the conditions of Theorem \ref{existence} hold. Then, for any continuous version of $X_\alpha$, almost surely
\begin{equation}\label{dimensionresult}
	\dim_{\mathcal{H}} G_{X_\alpha} \big( [0,1]^d \big) = \dim_{\mathcal{B}} G_{X_\alpha} \big( [0,1]^d \big) = d+1 - \min_{1 \leq j \leq m} \frac{H_j}{a_{p_j}^j},
\end{equation}
where, as above,
$$ G_{X_\alpha} \big( [0,1]^d \big) = \Big\{ \big( x, X_\alpha (x) \big) : x \in [0,1]^d \Big\} $$
is the graph of $X_\alpha$ over $[0,1]^d$.
\end{theorem}

\noindent \textit{Proof.} Let us choose a continuous version of $X_\alpha$ by Corollary \ref{exponent}. From Corollary \ref{exponent}, for any $0 < s < \min_{1 \leq j \leq m} \frac{H_j}{a_{p_j}^j}$, the sample paths of $X_\alpha$ satisfy almost surely a uniform H\"older condition of order $s$ on $[0,1]^d$. Thus, by Lemma \ref{upperbound}, we have
\begin{align*}
\dim_{\mathcal{H}} G_{X_\alpha} \big( [0,1]^d \big) = \overline{\dim_{\mathcal{B}}} G_{X_\alpha} \big( [0,1]^d \big) \leq d+1 - s, \quad a.s.
\end{align*}
Letting $s \uparrow \min_{1 \leq j \leq m} \frac{H_j}{a_{p_j}^j}$ along rational numbers yields the upper bound in \eqref{hausdorff}. 

It remains to prove the lower bound in \eqref{hausdorff}. Since the inequality
$$\underline{\dim_{\mathcal{B}}}  B\geq \dim_{\mathcal{H}} B$$
holds for any $B \subset \rd$ (see \cite[Chapter 3.1]{Fal}), it suffices to show
$$ \dim_{\mathcal{H}} G_{X_\alpha} \big( [0,1]^d \big) \geq d+1 - \min_{1 \leq j \leq m} \frac{H_j}{a_{p_j}^j}, \quad a.s. $$
Further, note that, since $Q = [ \frac{1}{2} , 1]^d \subset [0,1]^d$, we have
$$ \dim_{\mathcal{H}} G_{X_\alpha} \big( [0,1]^d \big) \geq  \dim_{\mathcal{H}} G_{X_\alpha} (Q) $$
by monotonicity of the Hausdorff dimension. Thus, it is even enough to show that
\begin{equation}\label{boundq}
\dim_{\mathcal{H}} G_{X_\alpha} (Q)  \geq d+1 - \min_{1 \leq j \leq m} \frac{H_j}{a_{p_j}^j}, \quad a.s.
\end{equation}
We will show this by combining the methods used in \cite{Aych,BMS, BL}. From now on, without loss of generality, we will assume that
$$\min_{1 \leq j \leq m} \frac{H_j}{a_{p_j}^j} = \frac{H_1}{a_{p_1}^1} .$$

Let $\gamma > 1$. According to \eqref{frostman}, it suffices to show that
$$I_\gamma := \int_{ Q \times Q } \mathbb{E} \Big[ \big( \| x-y \|^2 + |X_\alpha(x) - X_\alpha(y)|^2 \big) ^{-\frac{\gamma}{2}} \Big] dx dy < \infty$$
in order to obtain $\dim_{\mathcal{H}} G_{X_\alpha} (Q)  \geq \gamma$ almost surely.

Using the characteristic function of the symmetric $\alpha$-stable random field $X_\alpha$, as in the proof of \cite[Proposition 5.7]{BL}, it can be shown that there is a constant $C_1 >0$ such that
\begin{equation}\label{constant1}
I_\gamma \leq C_1 \int_{Q \times Q} \| x-y \|^{1-\gamma} \sigma^{-1} (x,y) dx dy,
\end{equation}
where
$$ \sigma (x,y) = \mathbb{E} \big[ | X_\alpha (x) - X_\alpha (y) |^\alpha \big] ^{\frac{1}{\alpha}} = \| X_\alpha (x) - X_\alpha (y) \|_\alpha.$$
Using the notation $ x= (x_1, \ldots, x_m ) \in \mathbb{R}^{d_1} \times \ldots \times \mathbb{R}^{d_m}$ for any vector $x \in \rd$ define
\begin{equation}\label{wl}
W_l = \{ (x,y) \in Q \times Q : x_1 \neq y_1 , 2^{-l-1} \leq \| x-y \| \leq 2^{-l} \}
\end{equation}
for all $l \in \mathbb{N}_0$. Then, using \eqref{constant1} and \eqref{wl} we obtain
\begin{align*}
I_\gamma & \leq C_1 \sum_{l=0}^\infty \int_{W_l} \| x-y \| ^{1-\gamma} \sigma^{-1} (x,y) dx dy \\
& \leq \frac{C_1}{2} \sum_{l=0}^\infty \int_{W_l} 2 ^{-l (1-\gamma)} \sigma^{-1} (x,y) dx dy .
\end{align*}
Therefore, the proof of the lower bound in Theorem \ref{hausdorff} follows from the following Proposition, by letting $\varepsilon \downarrow 0$. \hfill $\Box$
\begin{prop}\label{sigmawl}
Let $\varepsilon > 0$ be small enough. Then, for all $\gamma \in (1, d+1 - \frac{H_1}{a_{p_1}^1} - \varepsilon )$, we have
$$ \sum_{l=0}^\infty 2 ^{l (\gamma-1)}  \int_{W_l} \sigma^{-1} (x,y) dx dy < \infty .$$
\end{prop}

We will need several results in order to prove Proposition \ref{sigmawl}. The following Lemma (and its proof) can be seen as a generalization of \cite[Lemma 3.5]{Aych}.
\begin{lemma}\label{lemmaayache}
For any $x \in \rd$, let $x= (x_1, \ldots, x_m) \in \mathbb{R}^{d_1} \times \ldots \times \mathbb{R}^{d_m}$. Then there exists a constant $C_3 > 0$ such that for all $\beta \in (0,1)$ and for all $l \in \mathbb{N}_0$ we have
$$ \int_{W_l} \| x_1-y_1 \| ^{-\beta} dx_1 \ldots dx_m dy_1 \ldots dy_m \leq C_3 \cdot 2^{-l (d-\beta)} .$$

\end{lemma}

\begin{proof}
Define
$$G_l^j = \{ (s,t) \in [0,1]^{d_j} \times [0,1]^{d_j} : \| s-t\| \leq 2^{-l} \}$$
for all $l \in \mathbb{N}_0$ and $j =1, \ldots, m$. Then it follows from \eqref{wl} that for any $(x,y) \in \mathbb{R}^d \times \rd$, $x= (x_1, \ldots x_m), y =(y_1, \ldots, y_m) \in \mathbb{R}^{d_1} \times \ldots \times \mathbb{R}^{d_m}$, we have
$$\mathbbm{1}_{W_l} (x,y) \leq \prod_{j=1}^m \mathbbm{1}_{G_l^j} (x_j, y_j) .$$
From this, we obtain
\begin{align}\label{proof1}
& \int_{W_l} \| x_1-y_1 \| ^{-\beta} dx_1 \ldots dx_m dy_1 \ldots dy_m \\
& \leq \int_{G^1_l} \| x_1-y_1 \| ^{-\beta} dx_1 dy_1 \cdot \prod_{j=2}^m \int_{G_l^j} dx_j dy_j \nonumber.
\end{align}
By applying Fubini's Theorem and the inequaliy
$$ \| u - v \| ^{-\beta } \leq | u_1 - v_1 | ^{-\beta} $$
for any vectors $u = (u_1, \ldots, u_{d_1} )$ and $v = (v_1, \ldots, v_{d_1} ) \in \mathbb{R}^{d_1}$, it is easy to see that there are constants $c, c' >0$ such that
\begin{equation}\label{proof2}
\int_{G^1_l} \| x_1-y_1 \| ^{-\beta} dx_1 dy_1 \leq c \cdot(2^{-l} )^{d_1 - \beta} ,
\end{equation}
and
\begin{equation}\label{proof3}
\int_{G^j_l} dx_j dy_j \leq c' \cdot(2^{-l} )^{d_j} 
\end{equation}
for all $j =2, \ldots, m$. Combining \eqref{proof1}, \eqref{proof2} and \eqref{proof3} yields the statement of the Lemma.
\end{proof}

The following Theorem is crucial for proving Proposition \ref{sigmawl} and its proof is based on \cite[Theorem 1]{WuXiao}. See also \cite{Xiao1, Xiao2, Xiao3}.

\begin{theorem}\label{sigmabound}
There exists a constant $C_4 >0$, depending on $H_1, \ldots, H_m, q_1, \ldots, q_m$ and $d$ only, such that for all $x = (x_1, \ldots x_m), y =(y_1, \ldots y_m) \in [\frac{1}{2} , 1)^{d_1} \times \ldots \times [\frac{1}{2} , 1)^{d_m}$ we have
$$ \sigma (x,y) \geq C_4 \cdot \tau_{\tilde{E}_1} (x_1 - y_1)^{H_1} ,$$
where $\tau_{\tilde{E}_1} ( \cdot )$ is the radial part with respect to $\tilde{E}_1$.
\end{theorem}

\begin{proof}
Throughout this proof, we fix $x = (x_1, \ldots, x_m)$, $y =(y_1, \ldots y_m) \in [\frac{1}{2} , 1)^{d_1} \times \ldots \times [\frac{1}{2} , 1)^{d_m}$. We will show that
\begin{equation}\label{sigmatau}
\sigma (x,y) \geq C r^{H_1}
\end{equation}
for some constant $C>0$ independent of $x$ and $y$ and $r = \tau_{\tilde{E}_1} (x_1 - y_1)$. Without loss of generality we will assume that $r>0$, since for $r=0$ \eqref{sigmatau} always holds. By definition, we have
\begin{align}\label{sigmaalpha}
\sigma^\alpha (x,y) & = \mathbb{E} \big[  | X_\alpha (x) - X_\alpha (y) | ^\alpha \big] \nonumber \\
& = \int_\rd | \prod_{j=1}^m ( e^{i \langle x_j, \xi_j \rangle} - 1) - \prod_{j=1}^m ( e^{i \langle y_j, \xi_j \rangle} - 1) |^\alpha \prod_{j=1}^m | \psi_j (\xi_j) | ^{-\alpha H_j - q_j} d\xi .
\end{align}
Now, for every $j=1, \ldots, m$ we consider a so-called bump function $\delta_j \in C^{\infty} (\mathbb{R}^{d_j} )$ with values in $[0,1]$ such that $\delta_j (0) = 1$ and $\delta_j$ vanishes outside the open ball
$$B (K_j, 0) = \{ z \in \mathbb{R}^{d_j} : \tau_{\tilde{E}_j} (z) < K_j \} $$
for
\begin{align*}
K_j = \min \Big\{ 1, \frac{K_1^j}{K_2^j} ( \sqrt{d_1} \frac{1}{2}) ^{\frac{1}{a_1^j} - \frac{1}{a_{p_j}^j} + 2 \varepsilon} & , \frac{K_3^j}{K_4^j} ( \sqrt{d_1} \frac{1}{2}) ^{\frac{1}{a_{p_j}^j} - \frac{1}{a_{1}^j} - 2 \varepsilon} , \frac{K_1^j}{K_4^j}, \frac{K_3^j}{K_2^j}, \\
& K_1^j (  \sqrt{d_1} \frac{1}{2}) ^{\frac{1}{a_1^j} + \varepsilon} ,  K_3^j (  \sqrt{d_1} \frac{1}{2}) ^{\frac{1}{a_{p_j}^j} - \varepsilon} \Big\} ,
\end{align*}
where $\varepsilon >0$ is some (sufficiently) small number and $K_1^j , \ldots, K_4^j$ are the suitable constants derived from Lemma \ref{pcbounds} corresponding to the matrix $\tilde{E}_j$. The choice of the constant $K_j >0$ will be clear later in this proof. Let $\hat{\delta}_j$ be the Fourier transform of $\delta_j$. It can be verified that $\hat{\delta}_j \in C^\infty ( \mathbb{R}^{d_j} )$ as well and that $\hat{\delta}_j (\lambda_j)$ decays rapidly as $\| \lambda_j \| \to \infty$. By the Fourier inversion formula, we have
\begin{equation}\label{fourier1}
\delta_j (s_j) = \frac{1}{(2 \pi ) ^{d_j}} \int_{\mathbb{R}^{d_j} } e^{-i \langle s_j, \lambda_j \rangle} \hat{\delta}_j (\lambda_j) d\lambda_j
\end{equation}
for all $s_j \in \mathbb{R}^{d_j}$. Let $\delta_1^r (s_1) = \frac{1}{r^{q_1}} \delta_1 \big( (\frac{1}{r} )^{\tilde{E}_1} s_1 \big) .$ Then, by a change of variables in \eqref{fourier1}, for all $s_1 \in \mathbb{R}^{d_1}$ we obtain
\begin{equation}\label{fourier2}
\delta_1^r (s_1) = \frac{1}{(2 \pi ) ^{d_1}} \int_{\mathbb{R}^{d_1} } e^{-i \langle s_1, \lambda_1 \rangle} \hat{\delta}_1 (r^{\tilde{E}_1^T}\lambda_1) d\lambda_1 .
\end{equation}
Using Lemma \ref{pcbounds} and the fact that $\tau_{\tilde{E}_1} ( \cdot )$ is $\tilde{E}_1$-homogeneous, it is straightforward to see that $\tau_{\tilde{E}_j} (x_j) \geq K_j$, $\tau_{\tilde{E}_1} \big( (\frac{1}{r})^{\tilde{E}_1}(x_1 - y_1) \big) \geq K_1$ and $\tau_{\tilde{E}_1} \big( (\frac{1}{r})^{\tilde{E}_1}x_1 \big) \geq K_1$. Therefore, we have $\delta_1^r (x_1) = 0$, $\delta_1^r (x_1 - y_1) = 0$ and $\delta_j (x_j) = 0$ for all $j =2, \ldots, m$. Hence, combining this with \eqref{fourier1} and \eqref{fourier2} it follows that
\begin{align}\label{cauchy1}
I & : = \int_\rd \Big( \prod_{j=1}^m ( e^{i \langle x_j, \lambda_j \rangle} - 1) - \prod_{j=1}^m ( e^{i \langle y_j, \lambda_j \rangle} - 1) \Big) \nonumber \\
& \quad \cdot \prod_{j=1}^m  e^{- i \langle x_j, \lambda_j \rangle} \hat{\delta}_1 (r^{\tilde{E}_1^T} \lambda_1 ) \prod_{j=2}^m \hat{\delta}_j (\lambda_j) d \lambda \nonumber \\
& = (2 \pi )^d \Big( \delta_1^r (0) - \delta_1^r (x_1) \Big) \prod_{j=2}^m \Big( \delta_j(0) - \delta_j (x_j) \Big) \nonumber \\
& \quad - (2 \pi )^d \Big( \delta_1^r (x_1 - y_1) - \delta_1^r (x_1) \Big) \prod_{j=2}^m \Big( \delta_j(x_j-y_j) - \delta_j (x_j) \Big) \nonumber \\
& = (2 \pi)^d \frac{1}{r^{q_1}} .
\end{align}
Now, we first assume that $\alpha \geq 1$ and let $\beta>1$ be the constant such that $\frac{1}{\alpha} + \frac{1}{\beta} = 1$. By H\"older's inequality and \eqref{sigmaalpha}, we have
\begin{align}\label{cauchy2}
I & \leq \Big(  \int_\rd | \prod_{j=1}^m ( e^{i \langle x_j, \lambda_j \rangle} - 1) - \prod_{j=1}^m ( e^{i \langle y_j, \lambda_j \rangle} - 1) |^\alpha \prod_{j=1}^m | \psi_j (\lambda_j) | ^{-\alpha H_j - q_j} d\lambda \Big) ^{\frac{1}{\alpha}} \nonumber \\
& \quad \cdot \Big( \int_\rd \frac{1}{\big( \prod_{j=1}^m | \psi_j (\lambda_j) | ^{-\alpha H_j - q_j} \big)^{\frac{\beta}{\alpha}} } | \hat{\delta}_1  (r^{\tilde{E}_1^T} \lambda_1 ) \prod_{j=2}^m \hat{\delta}_j (\lambda_j) |^\beta d\lambda \Big) ^{\frac{1}{\beta}} \nonumber \\
& = \sigma (x,y) \cdot r^{-H_1 - \frac{q_1}{\alpha}- \frac{q_1}{\beta}} \cdot \Big( \int_\rd \frac{1}{\big( \prod_{j=1}^m | \psi_j (\lambda_j) | ^{-H_j - \frac{q_j}{\alpha}} \big)^{\beta} } \prod_{j=1}^m | \hat{\delta}_j (\lambda_j) |^\beta d\lambda \Big) ^{\frac{1}{\beta}} \nonumber \\
& = \tilde{C} \cdot \sigma (x,y) \cdot r^{-H_1 - q_1},
\end{align}
where $\tilde{C} >0$ is a constant, which only depends on $H_1, \ldots, H_m, q_1, \ldots, q_m, d$ and $\delta$. It is clear that \eqref{sigmatau} follows from \eqref{cauchy1} and \eqref{cauchy2}. If $\alpha \in (0,1)$, choose $k \in \mathbb{N}$ such that $k \alpha \geq 1$ and let $\beta' >1$ be the constant such that $\frac{1}{k\alpha} + \frac{1}{\beta'} = 1$. We first show that
\begin{align}\label{sigmak}
\Big(  \int_\rd | \prod_{j=1}^m ( e^{i \langle x_j, \lambda_j \rangle} - 1) & - \prod_{j=1}^m ( e^{i \langle y_j, \lambda_j \rangle} - 1) |^{k\alpha} \prod_{j=1}^m | \psi_j (\lambda_j) | ^{-\alpha H_j - q_j} d\lambda \Big) ^{\frac{1}{k\alpha}} \nonumber \\
 & \leq 2^{k \alpha (m+1)} \sigma(x,y)^{\frac{1}{k}} .
\end{align}
For $\lambda = (\lambda_1, \ldots, \lambda_m) \in \mathbb{R}^{d_1} \times \ldots \times \mathbb{R}^{d_m}$, let 
$$z (\lambda) = \prod_{j=1}^m ( e^{i \langle x_j, \lambda_j \rangle} - 1)  - \prod_{j=1}^m ( e^{i \langle y_j, \lambda_j \rangle} - 1)$$
and note that, since $| e^{it} -1 |^2 = 2 - 2 \cos t \leq 4$ for all $t \in \mathbb{R}$, it follows that
$$ |z (\lambda)| \leq \prod_{j=1}^m | e^{i \langle x_j, \lambda_j \rangle} - 1|  +  \prod_{j=1}^m | e^{i \langle y_j, \lambda_j \rangle} - 1| \leq 2^{m+1}.$$ From this, we obtain
\begin{align*}
& \Big(  \int_\rd | z (\lambda) |^{k\alpha} \prod_{j=1}^m | \psi_j (\lambda_j) | ^{-\alpha H_j - q_j} d\lambda \Big) ^{\frac{1}{k\alpha}} \\
& =   \Big(  \int_{\{\lambda \in \rd : | z (\lambda) | \leq 1 \} } | z (\lambda) |^{k\alpha} \prod_{j=1}^m | \psi_j (\lambda_j) | ^{-\alpha H_j - q_j} d\lambda  \\
& \quad + \int_{\{\lambda \in \rd : | z (\lambda) | > 1 \} } | z (\lambda) |^{k\alpha} \prod_{j=1}^m | \psi_j (\lambda_j) | ^{-\alpha H_j - q_j} d\lambda \Big) ^{\frac{1}{k\alpha}} \\
& \leq \Big(  \int_{\{\lambda \in \rd : | z (\lambda) | \leq 1 \} } | z (\lambda) |^{\alpha} \prod_{j=1}^m | \psi_j (\lambda_j) | ^{-\alpha H_j - q_j} d\lambda  \\
& \quad + \int_{\{\lambda \in \rd : | z (\lambda) | > 1 \} } | z (\lambda) |^{k\alpha + \alpha} \prod_{j=1}^m | \psi_j (\lambda_j) | ^{-\alpha H_j - q_j} d\lambda \Big) ^{\frac{1}{k\alpha}} \\
& \leq (2^{m+1})^{k \alpha}  \Big(  \int_\rd | z (\lambda) |^{\alpha} \prod_{j=1}^m | \psi_j (\lambda_j) | ^{-\alpha H_j - q_j} d\lambda \Big) ^{\frac{1}{k\alpha}} \\
& = (2^{m+1})^{k \alpha} \sigma (x,y)^{\frac{1}{k}}.
\end{align*}
Now, using \eqref{sigmak} and H\"older's inequality as before we obtain that
\begin{align}\label{cauchy3}
I & \leq \Big(  \int_\rd | \prod_{j=1}^m ( e^{i \langle x_j, \lambda_j \rangle} - 1) - \prod_{j=1}^m ( e^{i \langle y_j, \lambda_j \rangle} - 1) |^{k\alpha} \prod_{j=1}^m | \psi_j (\lambda_j) | ^{-\alpha H_j - q_j} d\lambda \Big) ^{\frac{1}{k\alpha}} \nonumber \\
& \quad \cdot \Big( \int_\rd \frac{1}{\big( \prod_{j=1}^m | \psi_j (\lambda_j) | ^{-\alpha H_j - q_j} \big)^{\frac{\beta'}{k\alpha}} } | \hat{\delta}_1  (r^{\tilde{E}_1^T} \lambda_1 ) \prod_{j=2}^m \hat{\delta}_j (\lambda_j) |^{\beta'} d\lambda \Big) ^{\frac{1}{\beta'}} \nonumber \\
& \leq 2^{k \alpha (m+1)} \sigma (x,y)^{\frac{1}{k}} \cdot r^{-\frac{H_1}{k} - \frac{q_1}{k\alpha}- \frac{q_1}{\beta'}} \nonumber \\
& \quad \cdot \Big( \int_\rd \frac{1}{\big( \prod_{j=1}^m | \psi_j (\lambda_j) | ^{-H_j - \frac{q_j}{\alpha}} \big)^{\beta'} } \prod_{j=1}^m | \hat{\delta}_j (\lambda_j) |^{\beta'} d\lambda \Big) ^{\frac{1}{\beta'}} \nonumber \\
& = \tilde{\tilde{C}} \cdot \big( \sigma (x,y) \cdot r^{-H_1 - kq_1} \big) ^{\frac{1}{k}},
\end{align}
where $\tilde{\tilde{C}} >0$ is a constant, which only depends on $H_1, \ldots, H_m, q_1, \ldots, q_m, k, \alpha, d$ and $\delta$. It is clear that \eqref{sigmatau} follows from \eqref{cauchy1} and \eqref{cauchy3}. This finishes the proof of the Theorem.
\end{proof}

We are now able to give a proof to Proposition \ref{sigmawl}. \newline \newline
\noindent \textit{Proof of Proposition \ref{sigmawl}.} Following Section 2.2, let $V_1, \ldots, V_{p_1}$ denote the spectral decomposition of $\mathbb{R}^{d_1}$ with respect to $\tilde{E}_1$ and let $O_i = V_1 \oplus \ldots \oplus V_i$ for $i \leq p_1$. We first assume that $p_1 = 1$. Let $C_5, C_6, C_7$ denote unspecified positive constants.  Since $p_1 = 1$, by Lemma \ref{pcbounds} one can find a constant $K>0$ such that for $\varepsilon >0$ (small) and $\| x_1 - y_1 \| \leq 2^{-l}$ we have
\begin{align}\label{taulemma}
\tau_{\tilde{E}_1 } (x_1 - y_1)^{-H_1} \leq \| x_1 - y_1 \| ^{-\frac{H_1}{a_{p_1}^1} - \varepsilon} .
\end{align}
Finally, using Theorem \ref{sigmabound}, \eqref{taulemma} and Lemma \ref{lemmaayache} we obtain
\begin{align*}
& \sum_{l=0}^\infty 2^{-l (1 - \gamma)} \int_{W_l} \sigma^{-1} (x,y) dx dy \\
&  \leq C_5 \sum_{l=0}^\infty 2^{-l (1 - \gamma)} \int_{W_l} \tau_{\tilde{E}_1 } (x_1 - y_1)^{-H_1} dx dy \\
& \leq C_6 \sum_{l=0}^\infty 2^{-l (1 - \gamma)} \int_{W_l} \| x_1 - y_1 \| ^{-\frac{H_1}{a_{p_1}^1} - \varepsilon} dx dy \\
& \leq C_6 \sum_{l=0}^\infty 2^{-l (1 - \gamma + d -\frac{H_1}{a_{p_1}^1} - \varepsilon)} < \infty ,
\end{align*}
since $\gamma < d+1 -\frac{H_1}{a_{p_1}^1} - \varepsilon$.

Now, assume that $p_1 \geq 2$. We will show that
$$J := \int_{W_l} \tau_{\tilde{E}_1 } (x_1 - y_1)^{-H_1} dx dy \leq \hat{C} \cdot 2^{-l (d-\frac{H_1}{a_{p_1}^1} - \varepsilon)} ,$$
for some constant $\hat{C} >0$ and some sufficiently small $\varepsilon >0$ in order to obtain the statement of Proposition \ref{sigmawl} from Theorem \ref{sigmabound} as above.

For $z = x_1 - y_1$, let us write $z = z_{p_1} + z_{p_1 - 1}$ for some $z_{p_1} \in V_{p_1}$ and $z_{p_1-1} \in O_{p_1-1}.$ Note that, since $V_{p_1}$ and $O_{p_1-1}$ are orthogonal in the chosen inner product, we have that $ \| z\| \leq 2^{-l}$ implies both $\| z_{p_1} \| \leq 2^{-l}$ and $\| z_{p_1-1} \| \leq 2^{-l}$. As before, let $C_7, \ldots, C_{10}$ denote unspecified positive constants. By a version of Lemma \ref{pcbounds}, restricted to the subspaces $V_{p_1}$ and $O_{p_1-1}$ respectively, as in the proof of \cite[Theorem 5.6]{BMS}, one can find a constant $\hat{K} >0$ such that for $\varepsilon >0 $ (small) we have
$$\tau_{\tilde{E}_1 } (z) \geq \hat{K} \Big( \| z_{p_1} \| ^{\frac{H_1}{a_{p_1}^1} + \varepsilon} + \| z_{p_1-1} \| ^{\frac{H_1}{a_{1}^1} + \varepsilon} \Big),$$
and from this we get
\begin{align*}
J & \leq C_7 \int_{\| z_{p_1} \| \leq 2^{-l} } \int_{\| z_{p_1-1} \| \leq 2^{-l} } \Big( \| z_{p_1} \| ^{\frac{H_1}{a_{p_1}^1} + \varepsilon} + \| z_{p_1-1} \| ^{\frac{H_1}{a_{1}^1} + \varepsilon} \Big) ^{-1} d z_{p_1} d z_{p_1-1} \\
& \quad \cdot \prod_{j=2}^m \int_{G_l^j} dx_j dy_j ,
\end{align*}
where the sets $G_l^j$ are defined as in the proof of Lemma \ref{lemmaayache}. Let $k = \dim V_{p_1}$ and observe that in the present case we have $1 \leq k \leq d_1 -1$. By using polar coordinates for both $V_{p_1}$ and $O_{p_1-1}$, we get that
\begin{align*}
J & \leq C_8 \int_0^{2^{-l}} \int_0^{2^{-l}} \Big( u ^{\frac{H_1}{a_{p_1}^1} + \varepsilon} + v ^{\frac{H_1}{a_{1}^1} + \varepsilon} \Big) ^{-1} u^{k-1} v^{d_1-k-1} du dv \cdot (2^{-l} ) ^{d-d_1} .
\end{align*}
The change of variables $u = tv$ and elementary integration yield
\begin{align*}
J & \leq C_8 \int_0^{2^{-l}} \int_0^{\frac{2^{-l}}{v}} \Big( (tv) ^{\frac{H_1}{a_{p_1}^1} + \varepsilon} + v ^{\frac{H_1}{a_{1}^1} + \varepsilon} \Big) ^{-1} (tv)^{k-1} v^{d_1-k-1} v dt dv \cdot (2^{-l} ) ^{d-d_1} \\
& \leq C_8 \int_0^{2^{-l}} \int_0^{\frac{2^{-l}}{v}} v^{- \frac{H_1}{a_{p_1}^1} - \varepsilon} \Big( t ^{\frac{H_1}{a_{p_1}^1} + \varepsilon} \Big) ^{-1} t^{k-1} v^{d_1-1} dt dv \cdot (2^{-l} ) ^{d-d_1} \\
& = C_8 \int_0^{2^{-l}}  v^{d_1 - 1 - \frac{H_1}{a_{p_1}^1} - \varepsilon} \int_0^{\frac{2^{-l}}{v}} t ^{k - \frac{H_1}{a_{p_1}^1} - \varepsilon - 1} dt dv \cdot (2^{-l} ) ^{d-d_1} \\
& = C_9 (2^{-l} ) ^{- \frac{H_1}{a_{p_1}^1} - \varepsilon + k} \int_0^{2^{-l}}  v^{d_1 - 1 - \frac{H_1}{a_{p_1}^1} - \varepsilon} v^{\frac{H_1}{a_{p_1}^1} + \varepsilon - k} dv \cdot (2^{-l} ) ^{d-d_1} \\
& = C_{10} (2^{-l}) ^{- \frac{H_1}{a_{p_1}^1} - \varepsilon + k}  (2^{-l}) ^{d_1-k} (2^{-l} ) ^{d-d_1} \\
& = C_{10} (2^{-l}) ^{d - \frac{H_1}{a_{p_1}^1} - \varepsilon} 
\end{align*}
and this finishes the proof of the Proposition. \hfill $\Box$
\newline

As an immediate consequence of Theorem \ref{hausdorff}, we obtain the following.

\begin{cor}\label{hoelderresult}
Let the assumptions of Theorem \ref{existence} hold. Then any continuous version of $X_\alpha$ admits $\min_{1 \leq j \leq m} \frac{H_j}{a_{p_j}^j}$ as the H\"older critical exponent.
\end{cor}

\begin{remark}
Let $\alpha = 2, d_j = \tilde{E}_j = 1$ for all $j=1, \ldots, m$ and consider the function $\psi (\xi_j) = | \xi_j |$ for all $\xi_j \in \mathbb{R}$. Clearly, $\psi_j$ is a homogeneous function and satisfies $\psi_j (\xi_j ) \neq 0$ for all $\xi_j \neq 0$. Thus, by Theorem \ref{existence}, we can define
$$X_2 (x) = \operatorname{Re} \int_\rd \prod_{j=1}^d ( e^{i x_j \xi_j} -1 ) |\xi_j|^{-H_j -\frac{1}{2}} W_2 (d \xi),  \quad x= (x_1, \ldots, x_d) \in \rd,$$
for all $0<H_j<1, j=1, \ldots, d$ and the statement of Theorem \ref{hausdorff} becomes
$$ \dim_{\mathcal{H}} G_{X_2} \big( [0,1]^d \big) = \dim_{\mathcal{B}} G_{X_2} \big( [0,1]^d \big) = d+1 - \min_{1 \leq j \leq d} H_j, \quad a.s.$$
Further, up to a multiplicative constant, the random field $X_2$ is a fractional Brownian sheet with Hurst indices $H_1, \ldots, H_d$ (see \cite{Herbin}). Thus, Theorem \ref{hausdorff} can be seen as a generalization of \cite[Theorem 1.3]{Aych}. Further, as noted above, for $m=1, d=d_1$ and $E=E_1 = \tilde{E}_1$ the random field $X_\alpha$ given by \eqref{harmonizable} coincides with the random field in \cite[Theorem 4.1]{BMS} and the statement of Theorem \ref{hausdorff} becomes
$$ \dim_{\mathcal{H}} G_{X_\alpha} \big( [0,1]^d \big) = \dim_{\mathcal{B}} G_{X_\alpha} \big( [0,1]^d \big) = d+1 - \frac{H_1}{a^1_{p_1}}, \quad a.s.$$
which is the statement of \cite[Theorem 5.6]{BMS} for $\alpha=2$ and \cite[Proposition 5.7]{BL} for $\alpha \in (0,2)$. We finally remark that Theorem \ref{hausdorff} can be proven similarly, if we replace $[0,1]^d$ in \eqref{dimensionresult} by any other compact cube of $\rd$.
\end{remark}

\bibliographystyle{plain}

\end{document}